\documentclass{amsart}

\usepackage[inactive]{srcltx} 
\usepackage{amsmath, amsthm, amscd, amsfonts, amssymb, graphicx, color, extarrows}
 \newtheorem{thm}{Theorem}[section]
 \newtheorem{cor}[thm]{Corollary}
 \newtheorem{lem}[thm]{Lemma}
 \newtheorem{prop}[thm]{Proposition}
 \theoremstyle{definition}
 
 \theoremstyle{remark}
 \newtheorem{rem}[thm]{Remark}
 \newtheorem{ex}[thm]{Example}
 \numberwithin{equation}{section}
\newcommand{\X}{\mathcal{X}}
\newcommand{\A}{ \mathcal{A}}
\newcommand{\B}{ \mathcal{B}}
\newcommand{\T}{ \mathcal{T}}
\newcommand{\M}{ \mathcal{M}}
\newcommand{\rr}{\mathcal{R}}
\newcommand{\ps}{\psi:\mathcal{A}\rightarrow \mathcal{X}}

\begin{document}

\title[Left ideal preserving maps on triangular algebras ]
 {Left ideal preserving maps on triangular algebras}

\author{Hoger Ghahramani}
\thanks{{\scriptsize
\hskip -0.4 true cm \emph{MSC(2010)}: 15A86; 16S50; 16D99; 16S99.
}}

\address{Department of
Mathematics, University of Kurdistan, P. O. Box 416, Sanandaj,
Iran.}

\email{h.ghahramani@uok.ac.ir; hoger.ghahramani@yahoo.com}


\address{}

\email{}

\thanks{}

\thanks{}



\date{}

\dedicatory{}

\commby{}

\begin{abstract}
Let $\A$ be a unital algebra over a commutative unital ring $\rr$. We say that $\A$ is a \textbf{SLIP} algebra if every $\rr$-linear map on $\A$ that leaves invariant every left ideal of $\A$ is a left multiplier. In this paper we study whether a triangular algebra $\begin{pmatrix}
  \A & \M \\
  0 & \B
\end{pmatrix}$ is a \textbf{SLIP} algebra and give some necessary or sufficient conditions for a triangular algebra to be a \textbf{SLIP} algebra, and various examples are given which illustrate limitations on extending some of the theory developed. Then our results are applied to generalized triangular matrix algebras and block upper triangular matrix algebras. Also, some \textbf{SLIP} algebras other than triangular algebras are provided.
\end{abstract}

\maketitle
\section{ Introduction}
Throughout this article $\rr$ will denote a commutative ring with unity and all algebras are associative over $\rr$ with unity and all modules are unital, unless otherwise is specified. The unity of an algebra is denoted by $1$. Let $\A$ be an algebra and $\X$ be a right $\A$-module. Recall that an $\rr$-linear map $\ps$ is said to be a \textit{left multiplier} if $\psi(a)=\psi(1)a$ for all $a\in\A$. It is called a \textit{local left multiplier} if for any $a\in\A$ there exists an element $x_{a}\in\X$ such that $\psi(a)=x_{a}a$. Clearly, each left multiplier is a local left multiplier. The converse is, in general, not true. According to \cite{had} we say that an $\rr$-linear map $\psi:\A\rightarrow \A$ is \textbf{LIP} (left-ideal-preserving) if $\psi(\mathcal{J})\subseteq \mathcal{J}$ for any left ideal $\mathcal{J}$ of $\A$. It is easily verified that the $\rr$-liear map $\psi:\A\rightarrow \A$ is \textbf{LIP} if and only if $\psi$ is a local left multiplier. So it is clear that any left multiplier $\psi:\A\rightarrow \A$ is \textbf{LIP} map, but the converse is not necessarily true (some examples will be given). It is natural an interesting to ask for which algebras any \textbf{LIP} map is a left multiplier, so we have the next definition. The algebra $\A$ is \textbf{SLIP} provided any \textbf{LIP} map is a left multiplier (the notion \textbf{SLIP} has been used in \cite{had}).
\par 
In the case that $\rr$ is a field, to say that $\A$ is \textbf{SLIP} is the same as saying that the algebra of left multipliers on $\A$ is algebraically reflexive \cite{had1}. Reflexivity (algebraically or topologically) is an important part of operator theory, and have been studied in ring theory and Banach algebra theory by several authors. Johnson \cite{john} has shown that if $\A$ is a semisimple Banach algebra with an approximate identity and $\psi:\A\rightarrow \A$ is a bounded operator that leaves invariant all closed left ideals of $\A$, then $\psi$ is a left multiplier of $\A$. Hadwin and Li \cite{had2} have shown that Johnson’s Theorem holds for all CSL algebras. In particular Hadwin, Li and their
collaborators \cite{had3, had, had2, had4, li} have been investigating questions of this type for the past 20 years for various reflexive operator algebras. Recently Katsoulis \cite{ka} has studied the reflexivity of left multipliers over certain operator algebras. In the purely algebraic, Bre$\breve{s}$ar, $\breve{S}$emrl and others have been investigating local multipliers in various other settings \cite{br1, br2}. Also Hadwin in \cite{had} has studied various \textbf{SLIP} algebras and \textbf{SLIP} algebras (and reflexivity) in ring theory are studied in \cite{ful1,ful2,ful3, had10,had11, sn1,sn2,sn3}. On the other hand recently, there has been a growing interest in the study of preserving linear maps on triangular algebras (for example linear maps that preserve zero products, Jordan products, commutativity, etc. and derivable, Jordan derivable, Lie derivable maps at zero point, etc. for instance, see \cite{An, Ben2,liu,  Zh0, Zha} and the references therein. Motivated by the concepts and results above, we will study whether a triangular algebra is a \textbf{SLIP} algebra and give some sufficient conditions under which a triangular algebra is a \textbf{SLIP} algebra. Our results are then applied to generalized triangular matrix algebras and block upper triangular matrix algebras. Other \textbf{SLIP} algebras are also provided. 
\par 
This article is organized as follows. In section 2 some preliminaries including the introduction of triangular algebras, generalized triangular matrix algebras and block upper triangular matrix algebras, are given. In section 3 we firstly study the relation between zero product determined algebras and \textbf{SLIP} algebras. By using this relationship we provide characterizations of the \textbf{SLIP} property for several classes of algebras. Then by study the \textbf{LIP} maps on triangular algebras we obtain a necessary condition and some sufficient conditions for a triangular algebra to be a \textbf{SLIP} algebra. Also examples are given which illustrate limitations on extending some of the theory developed. In section 4 the results in previous section are applied to generalized triangular matrix algebras and block upper triangular matrix algebras. Indeed, we prove that under certain conditions the generalized triangular matrix algebras are \textbf{SLIP} algebras, and we apply these results to block upper triangular matrix algebras. 
\section{ Preliminaries}
Recall that a \emph{triangular algebra} $Tri(\A,\M,\B)$ is an
algebra of the form
\[ Tri(\A,\M,\B):=\bigg \{ \begin{pmatrix}
  a & m \\
  0 & b
\end{pmatrix}\,\bigg | \, a \in \A,\, b\in \B, \, m\in \M\bigg \} \]
under the usual matrix operations, where $\A$ and $\B$ are
unital algebras and $\M$ is an $(\A, \B)$-bimodule. The most important
examples of triangular algebras are upper triangular matrices over
an algebra $\A$, block upper triangular matrix algebras, nest
algebras over a real or a complex Banach space $\X$ or a
Hilbert space $\mathcal{H}$, respectively and generalized
triangular matrix algebras. 
\par
Let $\A$ be an algebra. Recall that an idempotent $e\in \A$ is \emph{left semicentral} in
$\A$ if $\A e=e\A e$ \cite{Bir1}. We use $\mathcal{S}_{l}(\A)$
for the sets of all left semicentral idempotents. As is well
known \cite{Ch}, a left semicentral idempotent $e$ induces a
2-by-2 triangular matrix representation on $\A$. In
particular $\A\cong Tri(e\A e, e\A (1-e), (1-e)\A (1-e))$,
where $e\A e$ and $(1-e)\A (1-e)$ are algebras over $\rr$ with
the addition and multiplication of $\A$, but different unities
($e$ and $1-e$ are the identity elements of $e\A e$ and
$(1-e)\A (1-e)$, respectively) and $e\A (1-e)$ is a unital
$(e\A e,(1-e)\A (1-e))$-bimodule. If
$\mathcal{S}_{l}(\A)=\{0, 1\}$, then we say $\A$ is
\emph{semicentral reduced}. For more information, we refer to
\cite{Bir2}.
\par 
We say $\A$ has a \emph{generalized triangular matrix
representation} if there exists a $\rr$-algebra isomorphism:
\[ \theta: \A \rightarrow \begin{pmatrix}
  \A_{11} & \A_{12} & \cdots & \A_{1n} \\
  0 & \A_{22} & \cdots & \A_{2n} \\
  \vdots & \vdots & \ddots  & \vdots \\
  0 & 0 & \cdots & \A_{nn}
\end{pmatrix},\]
where each $\A_{ii}$ is an algebra with unity and $\A_{ij}$ is a
$(\A_{ii},\A_{jj})$-bimodule for $i< j$. An ordered set
$\{e_{1},\cdots, e_{n}\}$ of nonzero distinct idempotents in $\A$
is called a set of \emph{left triangulating idempotents} of $\A$
if all the following hold:
\begin{enumerate}
\item[(i)] $1=e_{1}+\cdots+e_{n}$;
\item[(ii)] $e_{1}\in \mathcal{S}_{l}(\A)$; and
\item[(iii)] $e_{k+1} \in \mathcal{S}_{l}(f_{k} \A f_{k})$, where $f_{k}=1-(e_{1}+\cdots+e_{k})$, for $1\leq k \leq n-1$ (see \cite{Bir2}).
\end{enumerate}
\begin{prop}\label{Bi}(\cite[Proposition 1.3]{Bir2})
$\A$ has a set of left triangulating idempotents if and only if $\A$ has a generalized triangular matrix representation.
\end{prop}
In fact, by this proposition if $\A$ has a set of left triangulating idempotents $\{e_{1},\cdots, e_{n}\}$, then we have
the $\rr$-algebra isomorphism:
\[ \A \cong \begin{pmatrix}
  e_{1}\A e_{1}& e_{1}\A e_{2}& \cdots & e_{1}\A e_{n} \\
  0 & e_{2}\A e_{2} & \cdots & e_{2}\A e_{n} \\
  \vdots & \vdots & \ddots  & \vdots \\
  0 & 0 & \cdots & e_{n}\A e_{n}
\end{pmatrix}.\]
Conversely, if $\A$ has a generalized triangular matrix
representation:
\[ \theta:\A \rightarrow \begin{pmatrix}
  \A_{11} & \A_{12} & \cdots & \A_{1n} \\
  0 & \A_{22} & \cdots & \A_{2n} \\
  \vdots & \vdots & \ddots  & \vdots \\
  0 & 0 & \cdots & \A_{nn}
\end{pmatrix},\]
then the set $\{\theta^{-1}(E_{1}),\cdots, \theta^{-1}(E_{n})\}$ is a set of left triangulating idempotents of $\A$, where $E_{k}$ is the $n$-by-$n$ matrix with the unity of $\A_{k}$ in the $(k,k)$-position and $0$ elsewhere.
\par
\begin{rem}\label{tt}
By the definition of a set of left triangulating idempotents and Proposition~\ref{Bi} we see that if $\A$ has a set of left
triangulating idempotents $\{e_{1},\cdots, e_{n}\}$, then the set $\{e_{2},\cdots, e_{n}\}$ is a set of left triangulating idempotents of $f\A f$, where $f=1-e_{1}$. Since $e_{1}\in \mathcal{S}_{l}(\A)$, it follows that $\A$ has a triangular matrix representation as $\A\cong Tri(e_{1}\A e_{1}, e_{1}\A f, f \A f)$. In this case we have the $\rr$-algebra isomorphism:
\[ f\A f \cong \begin{pmatrix}
    e_{2}\A e_{2} & \cdots & e_{2}\A e_{n} \\
   \vdots & \ddots  & \vdots \\
   0 & \cdots & e_{n}\A e_{n}
\end{pmatrix},\]
and the $(e_{1}\A e_{1}, f \A f)$-bimodule isomorphism:
\[ e_{1}\A f \cong (e_{1}\A e_{2}, \cdots , e_{1}\A e_{n}).\]
Also the set $\{e_{1},\cdots, e_{n-1}\}$ is a set of left triangulating idempotents of $f\A f$, where $f=1-e_{n}$. By \cite[Lemma 1.2]{Bir2}, $e_{j}\A e_{i}=\lbrace 0\rbrace$, for all $i< j\leq n$. So $f\in \mathcal{S}_{l}(\A)$ and hence $\A$ has a triangular matrix representation as $\A\cong Tri(f\A f, f\A e_{n}, e_{n}\A  e_{n})$. In this case we have the $\rr$-algebra isomorphism:
\[ f\A f \cong \begin{pmatrix}
    e_{1}\A e_{1} & \cdots & e_{1}\A e_{n-1} \\
   \vdots & \ddots  & \vdots \\
   0 & \cdots & e_{n-1}\A e_{n-1}
\end{pmatrix},\]
and the $(f\A f, e_{n}\A  e_{n})$-bimodule isomorphism:
\[f\A e_{n} \cong \begin{pmatrix}
e_{1}\A e_{n}\\
\vdots \\
e_{n-1}\A e_{n}
\end{pmatrix}.\]
\end{rem}
Let $M_{k\times m}(\A)$ denotes the set of all $k$-by-$m$
matrices over $\A$ (we denote $M_{k\times k}(\A)$ by $M_{k}(\A)$).
Let $\mathbb{N}$ be the set of all positive integers and let
$n\in \mathbb{N}$. For each positive integer $m$ with $m \leq
n$,we denote by $\bar{k} = (k_{1},\cdots, k_{m})\in
\mathbb{N}^{m}$ an ordered $m$-vector of positive integers such
that $n = k_{1}+\cdots+k_{m}$. The \emph{block upper triangular
matrix algebra} $B_{n}^{\bar{k}} (\A)$ is a subalgebra of
$M_{n}(\A)$ of the form
\[ B_{n}^{\bar{k}} (\A)\cong \begin{pmatrix}
  M_{k_{1}}(\A) & M_{k_{1}\times k_{2}}(\A) & \cdots & M_{k_{1}\times k_{m}}(\A) \\
  0 & M_{k_{2}}(\A) & \cdots & M_{k_{2}\times k_{m}}(\A) \\
  \vdots & \vdots & \ddots & \vdots \\
  0 & 0 & \cdots & M_{k_{m}}(\A)
\end{pmatrix}.\]
Note that $M_{n}(\A)$ is a special case of block upper triangular
matrix algebras. In particular, $B_{n}^{\bar{k}} (\A)=M_{n}(\A)$
if and only if $\bar{k}=(k_{1})$ with $k_{1}=n$.
\par
The block upper triangular matrix algebras $B_{n}^{\bar{k}} (\A)$
has a generalized triangular matrix representation and
$\{F_{1},\ldots,F_{m}\}$ is a set of left triangulating
idempotents of $B_{n}^{\bar{k}} (\A)$ such that
$F_{1}=\sum_{i=1}^{k_{1}}E_{i}$ and
$F_{j}=\sum_{i=1}^{k_{j}}E_{i+k_{1}+\cdots+k_{j-1}}$ for $2\leq
j\leq m$, where $E_{i}$ is the $n$-by-$n$ matrix with the unity
of $\A$ in the $(i,i)$-position and $0$ elsewhere. We have
$F_{j}B_{n}^{\bar{k}} (\A)F_{j}\cong M_{k_{j}}(\A)$ for any
$1\leq j\leq m$. 
\par 
Let $T_{n}(\A)$ be the algebra of all $n$-by-$n$ upper triangular
matrices with entries from $\A$. If we consider $\bar{k} =
(k_{1},\cdots, k_{n})\in \mathbb{N}^{n}$, whenever $k_{j}=1$ for
any $1\leq j\leq n$, then $T_{n}(\A)=B_{n}^{\bar{k}} (\A)$ is a
block upper triangular matrix algebra. In fact, $\{E_{1},\cdots,
E_{n}\}$ is a set of left triangulating idempotents of
$T_{n}(\A)$, where  $\A\cong E_{j}\A E_{j}$ for each $1\leq j \leq n$. 
\par 
The following terminology is used throughout this article. Let $\A$ be an algebra and $\A$ be a left $\A$-module, define the left annihilator of $\M$ by $l.ann_{\A}\M = \lbrace a\in \A \, \, : a\M = \lbrace 0\rbrace\rbrace$. Also, we employ lower case letters to denote elements in algebras and modules in the abstract setting and upper case letters to denote elements in triangular matrix algebras. $I$ stands for the identity element in matrix algebras and $1$ denotes the unity of algebras in general.
\section{\textbf{LIP} maps on triangular algebras}
Throughout this section $Tri(\A,\M,\B)$ denotes a triangular algebra, where $\A$ and $\B$ are algebras and $\M$ is an $(\A, \B)$-bimodule. In this section we study the \textbf{LIP} maps on triangular algebras and we obtain a necessary condition and some sufficient conditions to that a triangular algebra be a \textbf{SLIP} algebra. Firstly we check the relation between zero product determined algebras and \textbf{SLIP} algebras. 
\par 
The algebra $\A$ is called a \emph{zero product determined algebra} if for every $\rr$-module $\X$ and every $\rr$-bilinear map $\phi:\A \times \A \rightarrow \X$, the following holds: If $\phi(a,b)=0$ whenever $ab=0$, then there exists an $\rr$-linear map $L:\A^{2}\rightarrow \X$ such that $\phi(a,b)=L(ab)$ for all $a,b\in \A$. Note that since $\A$ is unital, it follows that $\A^{2}=\A$. The question of characterizing linear maps through zero products, etc. on algebras can be sometimes effectively solved by considering bilinear maps that preserve certain zero product properties (for instance, see \cite{gh} and the references therein). Motivated by these reasons  Bre$\check{\textrm{s}}$ar et al. \cite{bre11} introduced the concept of zero product determined algebras, which can be used to study the linear maps preserving zero product and derivable maps at zero point. We will see that the zero product determined algebras are \textbf{SLIP} algebras (and more). So by using this result we identify various \textbf{SLIP} algebras.
\begin{thm}\label{zpd}
Let $\A$ be a zero product determined algebra. Then for any right $\A$-module $\X$, every local left multiplier $\ps$ is a left multiplier.
\end{thm}
\begin{proof}
Define $\phi:\A \times \A \rightarrow \X$ by $\phi(a,b)=\psi(a)b$. So $\phi$ is an $\rr$-bilinear map. By hypothesis for any $a\in\A$, there is an element $x_{a}\in\X$ such that $\psi(a)=x_{a}a$. So for $a,b \in \A$ with $ab=0$, we have 
\[\phi(a,b)=\psi(a)b=x_{a}ab=0.\]
Since $\A$ is a zero product determined algebra, it follows that there exists an $\rr$-linear map $L:\A^{2}\rightarrow \X$ such that $\psi(a)b=\phi(a,b)=L(ab)$ for all $a,b\in \A$. Therefore $\psi(ab)=L(ab)=\psi(a)b$ for all $a,b\in \A$ and hence $\psi$ is a left multiplier.
\end{proof}
By this theorem it is clear that any zero product determined algebra is a \textbf{SLIP} algebra. We will see the converse of this result is not necessarily true. 
\par 
Bre$\breve{s}$ar showed that an algebra generated by its idempotents is a zero product
determined algebra \cite [Theorem 4.1] {bre3}. So from Theorem \ref{zpd} we have the following theorem.
\begin{thm}\label{gn}
Let $\A$ be an algebra which is generated by its idempotents and $\X$ be a right $\A$-module. If $\ps$ is a local left multiplier, then $\psi$ is a left multiplier. In particular $\A$ is a \textbf{SLIP} algebra.
\end{thm}
In the following corollary some classes of \textbf{SLIP} algebras are given which are generated by its idempotents.
\begin{cor}\label{mn}
Let $\A$ be any of the following algebras. Then for any right $\A$-module $\X$, every local left multiplier $\ps$ is a left multiplier. Indeed, $\A$ is a \textbf{SLIP} algebra.
\begin{enumerate}
\item[(i)] $\A=M_{n}(\B)$, where $\B$ is an algebra and $n\geq 2$.
\item[(ii)] $\A$ is a simple algebra containing a non-trivial idempotent.
\item[(iii)] $\A$ is an algebra containing an idempotent $e$ such that the ideal generated by $e$ and $1-e$, respectively, are both equal to $\A$. 
\end{enumerate}
\end{cor}
\begin{proof}
By \cite{br2} the algebra $\A$ is generated by its idempotents. The desired conclusion thus readily follows from Theorem \ref{gn}.
\end{proof}
It should be noted that the Corollary \ref{mn}-$(i)$  generalizes \cite[Theorem 3]{had}.
\par 
Let $Tri(\A,\M,\B)$ be a triangular algebra. In \cite[Theorem 2.1]{gh2} it is shown that $Tri(\A,\M,\B)$ is a zero product determined algebra if and only if $\A$ and $\B$ are zero product determined algebras. From this result and Theorem \ref{zpd}, we have the next proposition.
\begin{prop}\label{trzp}
Let $\T=Tri(\A,\M,\B)$ be a triangular algebra.
\begin{enumerate}
\item[(i)] If $\A$ and $\B$ are zero product determined algebras, $\X$ is a right $\T$-module, then every local left multiplier $\psi:\T\rightarrow \X$ is a left multiplier. So $\T$ is a \textbf{SLIP} algebra.
\item[(ii)] If $\T$ is a zero product determined algebra, $\X_{1}$ and $\X_{2}$ are right $\A$-module and right $\B$-module, respectively, then every local left multiplier $\psi_{1}:\A\rightarrow \X_{1}$ and $\psi_{2}:\B\rightarrow \X_{2}$ is a left multiplier. So $\A$ and $\B$ are \textbf{SLIP} algebras.
\end{enumerate}
\end{prop}
Now these questions are being raised: If the triangular algebra $Tri(\A,\M,\B)$ is a \textbf{SLIP} algebra, is it true that $Tri(\A,\M,\B)$ is a zero product determined algebra? If the triangular algebra $Tri(\A,\M,\B)$ is a \textbf{SLIP} algebra, are both of $\A$ and $\B$, \textbf{SLIP} algebras?  If for any right $\A$-module $\X_{1}$ and right $\B$-module $\X_{2}$, every local left multiplier $\psi_{1}:\A\rightarrow \X_{1}$ and $\psi_{2}:\B\rightarrow \X_{2}$ is a left multiplier, is $Tri(\A,\M,\B)$ a \textbf{SLIP} algebra? We will see that if $Tri(\A,\M,\B)$ is a \textbf{SLIP} algebra, it is not necessarily true that $\A$ is a \textbf{SLIP} algebra and so we obtain classes of \textbf{SLIP} triangular algebras which are not zero product determined algebras. Also we show that if $\A$ is a \textbf{SLIP} algebra and for any right $\B$-module $\X$, every local left multiplier $\psi:\B\rightarrow \X$ is a left multiplier, then $Tri(\A,\M,\B)$ is a \textbf{SLIP} algebra. So these result is a generalization of Proposition \ref{trzp}-$(i)$ (since \textbf{SLIP} algebras are not necessarily zero product determined algebras).
\par 
In the following lemma we describe the structure of \textbf{LIP} maps on triangular algebras.
\begin{lem}\label{lip}
Let $\T=Tri(\A,\M,\B)$ be a triangular algebra and $\psi:\T\rightarrow \T$ be a \textbf{LIP} map. Then there are $\rr$-linear maps $\alpha:\A\rightarrow \A$, $\tau:\M\rightarrow \M$, $\beta_{1}:\B\rightarrow \M$ and $\beta_{2}:\B \rightarrow\B$ such that 
\[\psi(\begin{pmatrix}
  a & m \\
  0 & b
\end{pmatrix})=\begin{pmatrix}
  \alpha(a) & \beta_{1}(b)+\tau(m) \\
  0 & \beta_{2}(b)
\end{pmatrix},\]
where $\alpha$ and $\beta_{2}$ are \textbf{LIP} maps, $\beta_{1}$ is a local left multiplier and $\tau(am)=\alpha(a)m$ for all $a\in\A$ and $m\in\M$.
\end{lem}
\begin{proof}
Since $\begin{pmatrix}
  \A & 0 \\
  0 & 0
\end{pmatrix}$, $\begin{pmatrix}
  0 & \M \\
  0 & 0
\end{pmatrix}$ and $\begin{pmatrix}
  0 & \M \\
  0 & \B
\end{pmatrix}$ are left ideals of $\T$ and $\psi$ is a \textbf{LIP} map, it follows that for every $a\in \A,b\in\B, m\in\M$
\[\psi(\begin{pmatrix}
  a & 0 \\
  0 & 0
\end{pmatrix})=\begin{pmatrix}
  \alpha(a) & 0\\
  0 & 0
\end{pmatrix}, \quad
\psi(\begin{pmatrix}
  0 & m \\
  0 & 0
\end{pmatrix})=\begin{pmatrix}
  0 & \tau(m)\\
  0 & 0
\end{pmatrix}\] and
\[\psi(\begin{pmatrix}
  0 & 0 \\
  0 & b
\end{pmatrix})=\begin{pmatrix}
  0 & \beta_{1}(b)\\
  0 & \beta_{2}(b)
\end{pmatrix}, \]
where $\alpha:\A\rightarrow \A$, $\tau:\M\rightarrow \M$, $\beta_{1}:\B\rightarrow \M$ and $\beta_{2}:\B \rightarrow\B$ are $\rr$-linear maps. The mapping $\psi$ is a local left multiplier, since it is a \textbf{LIP} map. So for $T=\begin{pmatrix}
  a & 0 \\
  0 & b
\end{pmatrix}\in \T$ ($a\in\A, b\in \B$), there is an element $X_{T}=\begin{pmatrix}
  a_{T} & m_{T} \\
  0 & b_{T}
\end{pmatrix}\in\T$, such that $\psi(T)=X_{T}T=\begin{pmatrix}
  a_{T}a & m_{T}b \\
  0 & b_{T}b
\end{pmatrix}$. Hence $\alpha(a)=a_{T}a$, $\beta_{1}(b)=m_{T}b$ and $\beta_{2}(b)=b_{T}b$ for all $a\in\A, b\in\B$. So these maps are local left multiplier. For any $a\in\A$ and $m\in\M$, letting 
$T=\begin{pmatrix}
  a & am \\
  0 & 0
\end{pmatrix}$ and $S=\begin{pmatrix}
  0 & -m \\
  0 & 1
\end{pmatrix}$. So $TS=0$ and there is an element $X_{T}\in\T$ such that $\psi(T)=X_{T}T$. Thus $\begin{pmatrix}
  \alpha(a) & \tau(am) \\
  0 & 0
\end{pmatrix}\begin{pmatrix}
  0 & -m \\
  0 & 1
\end{pmatrix}=\psi(T)S=X_{T}TS=0$. Therefore $\tau(am)=\alpha(a)m$ for all $a\in\A,m\in\M$.
\end{proof}
In the next theorem we obtain a necessary condition for that the triangular algebra $Tri(\A,\M,\B)$ be a \textbf{SLIP} algebra.
\begin{thm}\label{nec}
Let $\T=Tri(\A,\M,\B)$ be a \textbf{SLIP} triangular algebra. Then $\B$ is a \textbf{SLIP} algebra and every local left multiplier from $\B$ into $\M$ is a left multiplier.
\end{thm}
\begin{proof}
Suppose that $\beta_{1}:\B\rightarrow \M$ and $\beta_{2}:\B\rightarrow \B$ are local left multipliers and define the $\rr$-linear map $\psi:\T\rightarrow \T$ by $\psi(\begin{pmatrix}
  a & m \\
  0 & b
\end{pmatrix})=\begin{pmatrix}
  0 & \beta_{1}(b) \\
  0 & \beta_{2}(b)
\end{pmatrix}$. For each $b\in\B$, there are elements $c_{b}\in\B$ and $n_{b}\in\M$ such that $\beta_{1}(b)=n_{b}b$ and $\beta_{2}(b)=c_{b}b$. Now according to any $T=\begin{pmatrix}
  a & m \\
  0 & b
\end{pmatrix}\in \T$, letting $X_{T}=\begin{pmatrix}
  0 & n_{b} \\
  0 & c_{b}
\end{pmatrix}$. So we have
\[\psi(T)=\begin{pmatrix}
  0 & n_{b}b \\
  0 & c_{b}b
\end{pmatrix}=X_{T}T.\]
Hence $\psi$ is a \textbf{LIP} map and from hypothesis it is a left multiplier i.e., $\psi(T)=\psi(I)T$ for all $T\in \T$. So for all $b\in \B$ we see that
\[\begin{pmatrix}
  0 & \beta_{1}(b) \\
  0 & \beta_{2}(b)
\end{pmatrix}=\psi(\begin{pmatrix}
  0 & 0 \\
  0 & b
\end{pmatrix})=\begin{pmatrix}
  0 & \beta_{1}(1) \\
  0 & \beta_{2}(1)
\end{pmatrix}\begin{pmatrix}
  0 & 0 \\
  0 & b
\end{pmatrix}=\begin{pmatrix}
  0 & \beta_{1}(1)b \\
  0 & \beta_{2}(1)b
\end{pmatrix},\]
where $1$ is the unity of $\B$. Thus $\beta_{1}$ and $\beta_{2}$ are left multipliers.
\end{proof}
By this theorem we get a necessary condition for that an idempotent in a \textbf{SLIP} algebras be left semicentral.
\begin{cor}
Suppose $\A$ is a \textbf{SLIP} algebra. 
\begin{enumerate}
\item[(i)] If $a\in\A$ is a non-trivial idempotent ($e\neq 0,1$) which is left semicentral, then $(1-e)\A(1-e)$ is a \textbf{SLIP} algebra.
\item[(ii)] If for any non-trivial idempotent $e\in\A$, the algebra $(1-e)\A(1-e)$ is not a \textbf{SLIP} algebra, then $\A$ is semicentral reduced.
\end{enumerate}
\end{cor} 
In the following results we give some sufficient conditions to that a triangular algebra be a \textbf{SLIP} algebra.
\begin{thm}\label{suslip1}
Let $\T=Tri(\A,\M,\B)$ be a triangular algebra. Let $l.ann_{\A}\M =\lbrace 0 \rbrace$, $\B$ be a \textbf{SLIP} algebra and every local left multiplier from $\B$ into $\M$ is a left multiplier. Then $\T$ is a \textbf{SLIP} algebra.
\end{thm}
\begin{proof}
Suppose that $\psi:\T\rightarrow \T$ be a \textbf{LIP} map. By Lemma \ref{lip}
\[\psi(\begin{pmatrix}
  a & m \\
  0 & b
\end{pmatrix})=\begin{pmatrix}
  \alpha(a) & \beta_{1}(b)+\tau(m) \\
  0 & \beta_{2}(b)
\end{pmatrix},\]
where $\beta_{1}:\B\rightarrow \M$ and $\beta_{2}:\B\rightarrow \B$ are local left multipliers and $\tau(am)=\alpha(a)m$ for all $a\in\A$ and $m\in\M$. By hypothesis 
\[ \beta_{1}(b)=\beta_{1}(1)b, \quad \beta_{2}(b)=\beta_{2}(1)b \quad (b\in \B).\]
For every $a, a^{\prime}\in \A$ and $m\in \M$ we have
\[ \tau(aa^{\prime}m)=\alpha(aa^{\prime})m.\]
On the other hand,
\[ \tau(aa^{\prime}m)=\alpha(a)a^{\prime}m.\]
By comparing the two expressions for $\tau(aa^{\prime}m)$ and the hypothesis $l.ann_{\A}\M =\lbrace 0 \rbrace$, we arrive at $\alpha(aa^{\prime})=\alpha(a)a^{\prime}$. So
\[ \alpha(a)=\alpha(1)a, \quad \tau(m)=\alpha(1)m \quad (a\in\A, m\in \M).\]
Hence from these equations
\[ \psi(T)=\psi(I)T,\]
for all $T\in \T$.
\end{proof}
Let $\X$ be an $\rr$-module. It is obvious that each $\rr$-linear local left multiplier from $\rr$ into $\X$ is a left multiplier. So by Theorem \ref{suslip1} we have the next corollary.
\begin{cor}\label{c1}
Let $\A$ be an algebra (over $\rr$). Then the triangular algebra 
\[Tri(\A,\A,\rr)=\begin{pmatrix}
  \A & \A \\
  0 & \rr
\end{pmatrix}\]
is a \textbf{SLIP} algebra. 
\end{cor}
Let $\M$ be a right $\B$-module. Denote the algebra of all $\B$-module endomorphisms of $\M$ by $End_{\B} (\M)$. Let $\A=End_{\B} (\M)$. Then $\M$ is an $(\A,\B)$-bimodule equipped with $\phi · m := \phi(m)$ ($m \in \M, \phi\in \A$) and $l.ann_{\A}\M =\lbrace 0 \rbrace$. So by Theorem \ref{nec} and Theorem \ref{suslip1} we conclude the following corollary.
\begin{cor}\label{c2}
 Let $\M$ be a right $\B$-module. Then the triangular algebra 
\[Tri(End_{\B} (\M),\M,\B)=\begin{pmatrix}
  End_{\B} (\M) & \M \\
  0 & \B
\end{pmatrix}\]
is a \textbf{SLIP} algebra if and only if $\B$ is a \textbf{SLIP} algebra and every local left multiplier from $\B$ into $\M$ is a left multiplier.
\end{cor}
In the following theorem we don't require the condition $l.ann_{\A}\M =\lbrace 0 \rbrace$.
\begin{thm}\label{suslip2}
Let $\T=Tri(\A,\M,\B)$ be a triangular algebra. Let $\A$ and $\B$ be \textbf{SLIP} algebras and every local left multiplier from $\B$ into $\M$ is a left multiplier. Then $\T$ is a \textbf{SLIP} algebra.
\end{thm}
\begin{proof}
Suppose that $\psi:\T\rightarrow \T$ be a \textbf{LIP} map. By Lemma \ref{lip}
\[\psi(\begin{pmatrix}
  a & m \\
  0 & b
\end{pmatrix})=\begin{pmatrix}
  \alpha(a) & \beta_{1}(b)+\tau(m) \\
  0 & \beta_{2}(b)
\end{pmatrix},\]
where $\alpha:\A\rightarrow\A$ and $\beta_{2}:\B\rightarrow\B$ are \textbf{LIP} maps, $\beta_{1}:\B\rightarrow\M$ is a local left multiplier and $\tau:\M\rightarrow\M$ satisfies $\tau(am)=\alpha(a)m$ for all $a\in\A$ and $m\in\M$. By hypothesis 
\[ \alpha(a)=\alpha(1)a, \,\,\beta_{1}(b)=\beta_{1}(1)b, \,\, \beta_{2}(b)=\beta_{2}(1)b \,\, \text{and} \,\, \tau(m)=\alpha(1)m \quad ,\]
for all $a\in\A,b\in \B,m\in\M$. Therefore from these equations
\[ \psi(T)=\psi(I)T,\]
for all $T\in \T$.
\end{proof}
Now some examples are given which illustrate limitations on extending some of the theory developed and and these examples show that the classes of triangular algebras which are satisfying in conditions of Theorem \ref{suslip1} or Theorem \ref{suslip2} are different from each other. So we firstly need to obtain some algebras which are not \textbf{SLIP} algebras. 
\begin{rem}\label{div}
Every division \textbf{SLIP} algebra $\A$ is a field. Consider the arbitrary elements $a,b\in\A$ and define the $\rr$-linear map $\psi_{b}(a)=ab$. Since $\A$ is a division algebra, it follows that $\psi_{b}$ is a \textbf{LIP} map. So by the hypothesis that $\A$ is \textbf{SLIP}, we have  $ab=\psi_{b}(a)=\psi_{b}(1)a=ba$. Hence $\A$ is commutative. 
\end{rem}
From Remark \ref{div}, it is concluded that the quaternion algebra $\mathbb{H}(\mathbb{R})$ over the real field $\mathbb{R}$ is not a \textbf{SLIP} algebra. In the following example a non-division algebra is provided which is not \textbf{SLIP}. This example is given in \cite[Example 2.4]{ka}.
\begin{ex}\label{kak}
Let $\mathbb{F}$ be a field. Then
\[\mathcal{U}:=\bigg \{ \begin{pmatrix}
  \lambda & \mu \\
  0 & \lambda
\end{pmatrix}\,\bigg | \, \lambda,\mu \in\mathbb{F} \bigg \}\]
is an algebra over $\mathbb{F}$. This algebra is not \textbf{SLIP}. 
\end{ex}
Now by Corollary \ref{c1} and these examples we can get a \textbf{SLIP} triangular algebra $Tri(\A,\M,\B)$ such that $\A$ is not a \textbf{SLIP} algebra. Hence $Tri(\A,\M,\B)$ is not a zero product determined algebra (by Theorem \ref{zpd} and \cite[Theorem 2.1]{gh2}).
\begin{ex}
Let $\T$ be any of the following triangular algebras.
\[Tri( \mathbb{H}(\mathbb{R}) , \mathbb{H}(\mathbb{R}) ,  \mathbb{R}) \quad \text{or} \quad  Tri(\mathcal{U},\mathcal{U},  \mathbb{F}),\]
where $\mathcal{U}$ is the algebra described in Example \ref{kak}.
Then by Corollary \ref{c1}, $\T$ is a \textbf{SLIP} algebra but $\mathbb{H}(\mathbb{R})$ and $\mathcal{U}$ are not \textbf{SLIP} algebras.
\end{ex}
In this example $\T$ satisfies in conditions of Theorem \ref{suslip1} but doesn't satisfy in conditions of Theorem \ref{suslip2}. This example also shows that the converse of Theorem \ref{suslip2} is not necessarily true. 
\par 
The next example shows that the converse of Theorem \ref{suslip1} is not necessarily true.
\begin{ex}\label{e3}
Assume that $\A$ is a zero product determined algebra. By \cite[Proposition 2.8]{gh3}, $\A \times \A$ is a zero product determined algebra. By the following make $\A$ into an $((\A \times \A),\A)$-bimodule: 
\[(a,b)x:=ax \quad (a,b,x\in\A), \]
and the right multiplication is the usual multiplication of $\A$. By Proposition \ref{trzp}-$(i)$, the triangular algebra $\T=Tri(\A \times \A, \A, \A)$ is a \textbf{SLIP} algebra while $l.ann_{\A \times \A}\A=\lbrace 0 \rbrace \times \A \neq \lbrace 0 \rbrace $.
\end{ex}
In Example \ref{e3}, $\T$ satisfies in conditions of Theorem \ref{suslip2} but doesn't satisfy in conditions of Theorem \ref{suslip1}.
\par 
In the following example we show that the conditions on $\A$ in Theorems \ref{suslip1} and \ref{suslip2} are not superfluous.
\begin{ex}
Let $\A$ and $\B$ be algebras such that $\A$ is a zero product determined algebra and $\B$ is not a \textbf{SLIP} algebra. By the following make $\A$ into an $((\A \times \B),\A)$-bimodule: 
\[(a,b)x:=ax \quad (a,x\in\A,b\in\B), \]
and the right multiplication is the usual multiplication of $\A$. Now consider the triangular algebra $\T=Tri(\A \times \B,\A,\A)$ and we show that $\T$ is not a \textbf{SLIP} algebra. Since $\B$ is not a \textbf{SLIP} algebra, there is a local left multiplier $\rho:\B\rightarrow \B$ which is not a left multiplier. So for any $b\in \B$ there exist an element $c_{b}\in\B$, such that $\rho(b)=c_{b}b$. Define the $\rr$-linear map $\psi:\T\rightarrow \T$ by
\[ \psi(\begin{pmatrix}
  (a_{1},b) & a_{2} \\
  0 & a_{3}
\end{pmatrix})=\begin{pmatrix}
  (0,\rho(b)) & 0 \\
  0 & 0
\end{pmatrix}.\]
For any $T=\begin{pmatrix}
  (a_{1},b) & a_{2} \\
  0 & a_{3}
\end{pmatrix}\in \T$, there is $X_{T}=\begin{pmatrix}
  (0,c_{b}) & 0 \\
  0 & 0
\end{pmatrix}$ such that 
\[\psi(T)=X_{T}T.\]
So $\psi$ is a \textbf{LIP} map which is not a left multiplier. Otherwise, if $\psi(T)=\psi(I)T$ for all $T\in\T$, then $\rho$ is a left multiplier and this is a contradiction. In this example $l.ann_{\A \times \B}\A=\lbrace 0 \rbrace \times \B \neq \lbrace 0 \rbrace $ and by \cite[Lemma 5]{had}, $\A \times \B$ is not a \textbf{SLIP} algebra.
\end{ex}
In this example we can let the algebra $\A$ as any of the algebras in Corollary \ref{mn} and the algebra $\B$ as the quaternion algebra $\mathbb{H}(\mathbb{R})$ or the algebra $\mathcal{U}$ in Example \ref{kak}.
\section{Applications to generalized triangular matrix algebras and block upper triangular matrix algebras }
In this section our previous results in section 3 are applied to generalized triangular matrix algebras and block upper triangular matrix algebras. 
\begin{thm}\label{gn1}
Let the algebra $\A$ has a generalized triangular matrix representation 
\[  \A = \begin{pmatrix}
  \A_{11} & \A_{12} & \cdots & \A_{1n} \\
  0 & \A_{22} & \cdots & \A_{2n} \\
  \vdots & \vdots & \ddots  & \vdots \\
  0 & 0 & \cdots & \A_{nn}
\end{pmatrix},\]
where each $\A_{ii}$ is an algebra with unity and $\A_{ij}$ is a
$(\A_{i},\A_{j})$-bimodule for $i< j$. If any local left multiplier from $\A_{ii}$ ($1\leq i \leq n$) into $\A_{ki}$ ($1\leq k \leq i$) is a left multiplier, then $\A$ is a \textbf{SLIP} algebra.
\end{thm}
\begin{proof}
We denote the elements of $\A_{ij}$ by $a_{ij}$, $1_{i}$ for the unity of $\A_{ii}$ and $a_{ij}E_{ij}$ for the $n$-by-$n$ matrix with the $a_{ij}\in \A_{ij}$ at the $(i, j)$-entry and $0$ in all other entries. Also $E_{i}$ denotes the matrix $1_{i}E_{ii}$. The set $\lbrace E_{1},\cdots, E_{n}\rbrace$ is a set of left triangulating idempotents of $\A$ (by Proposition \ref{Bi}).
\par 
The proof is by induction on n. If $n=1$, then $\A=\A_{11}$ and the result is obvious in this case.
\par 
Assume $n\geq 2$ and for each algebra that has a set of left triangulating idempotents with $n-1$ elements, the result is true.
\par 
Let $\A$ has a set of left triangulating idempotents  $\lbrace E_{1},\cdots, E_{n}\rbrace$. By Remark \ref{tt}, $\A\cong Tri(E_{1}\A E_{1}, E_{1}\A F, F\A F)$ and $\lbrace E_{2},\cdots, E_{n}\rbrace$ is a set of left triangulating idempotents of $F\A F$, where $F=I-E_{1}=\sum_{i=2}^{n}E_{i}$ is the unity of $F\A F$. Also we have the $\rr$-algebra isomorphisms:
\[E_{1}\A E_{1}\cong \A_{11}, \quad F\A F \cong \begin{pmatrix}
    \A_{22} & \cdots & \A_{2n} \\
   \vdots & \ddots  & \vdots \\
   0 & \cdots & \A_{nn}
\end{pmatrix},\]
and the $(\A_{11},F\A F)$-bimodule isomorphism:
\[ E_{1}\A F \cong (\A_{12} , \cdots , \A_{1n}).\]
From the hypothesis $\A_{11}$ is a \textbf{SLIP} algebra. Also by induction hypothesis $F\A F$ is a \textbf{SLIP} algebra. Let $\psi: F\A F \rightarrow E_{1}\A F$ be a local left multiplier, we show that $\psi$ is a left multiplier. 
\par 
Since $\psi$ is an $\rr$-linear map, there exist $\rr$-linear maps $\phi_{ij}^{k}:\A_{ij}\rightarrow \A_{1k}$ such that
\[\psi(a_{ij}E_{ij})=\sum_{k=2}^{n}\phi_{ij}^{k}(a_{ij})E_{1k},\]
where $2\leq i \leq j \leq n$ and $2\leq k \leq n$. 
\par 
We complete the proof by checking some steps.
\\
\textit{Step 1.} $\phi_{ii}^{k}=0$ for all $2\leq i,k \leq n$ with $i\neq k$.
\\
\par 
For each $a_{ii}\in \A_{ii}$ ($2\leq i \leq n$), let $T=a_{ii}E_{i}$. Since $\psi$ is local left multiplier, there is $X_{T}\in E_{1}\A F$ such that $\psi(T)=X_{T}T$. Now $T(F-E_{i})=0$ and hence $\psi(T)(F-E_{i})=0$. So
\[0=(\sum_{k=2}^{n}\phi_{ii}^{k}(a_{ii})E_{1k})(F-E_{i})=\sum_{k=2}^{n}\phi_{ii}^{k}(a_{ii})E_{1k}-\phi_{ii}^{i}(a_{ii})E_{1i}.\]
Therefore for all $2\leq i,k \leq n$ with $i\neq k$, we have $\phi_{ii}^{k}=0$.
\\
\textit{Step 2.} $\phi_{ij}^{k}=0$ for all $2\leq i< j \leq n$ , $2\leq k \leq n$ with $j\neq k$.
\\
\par
For each $a_{ij}\in \A_{ij}$ $(2\leq i< j  \leq n)$, let $T=a_{ij}E_{ij}$. For any $2\leq k \leq n$ with $j\neq k$, we have $TE_{k}=0$. So by a similar argument as in Step 1 $\psi(T)E_{k}=0$. Hence
\[ 0=(\sum_{l=2}^{n}\phi_{ij}^{l}(a_{ij})E_{1l})E_{k}=\phi_{ij}^{k}(a_{ij})E_{1k}.\]
The result now follows from the above equation.
\\
\textit{Step 3.} $\phi_{ij}^{j}(a_{ij})E_{1j}=\phi_{ii}^{i}(1_{i})E_{1i}a_{ij}E_{ij}$ for all $2\leq i< j \leq n$ and $a_{ij} \in \A_{ij}$ .
\\
\par
For all $a_{ij} \in \A_{ij}$, we have $(E_{i}+a_{ij}E_{ij})(-a_{ij}E_{ij}+E_{j})=0$. So $\psi(E_{i}+a_{ij}E_{ij})(-a_{ij}E_{ij}+E_{j})=0$. By Steps 1, 2 we see that $\psi(E_{i})=\phi_{ii}^{i}(1_{i})E_{1i}$ and $\psi(a_{ij}E_{ij})=\phi_{ij}^{j}(a_{ij})E_{1j}$. So
\[ (\phi_{ii}^{i}(1_{i})E_{1i} +\phi_{ij}^{j}(a_{ij})E_{1j} )(-a_{ij}E_{ij}+E_{j})=0,\]
and hence
\[ \phi_{ij}^{j}(a_{ij})E_{1j}=\phi_{ii}^{i}(1_{i})E_{1i}a_{ij}E_{ij}.\]
\\
\textit{Step 4.} $\phi_{ii}^{i}(a_{ii})E_{1i}=\phi_{ii}^{i}(1_{i})E_{1i}a_{ii}E_{i}$ for all $2\leq i \leq n$ and $a_{ii} \in \A_{ii}$ .
\\
\par
For each $a_{ii}\in \A_{ii}$ ($2\leq i \leq n$), let $T=a_{ii}E_{i}$. Since $\psi$ is local left multiplier, there is $X_{T}=\sum_{k=2}^{n}x_{1k}E_{1k}$ such that $\psi(T)=X_{T}T$. So by Step 1 
\[\phi_{ii}^{i}(a_{ii})E_{1i} =x_{1i}a_{ii}E_{1i}\]
 and hence $\phi_{ii}^{i}$ is a local left multiplier. From hypothesis $\phi_{ii}^{i}(a_{ii})=\phi_{ii}^{i}(1_{i})a_{ii}$. Hence
\[\phi_{ii}^{i}(a_{ii})E_{1i}=\phi_{ii}^{i}(1_{i})E_{1i}a_{ii}E_{i}.\]
By Steps 1--4, it follows that $\psi(T)=\psi(F)T$ for all $T\in F\A F$. So $\psi$ is a left multiplier.
\par 
Now from Theorem \ref{suslip2}, it follows that the algebra $\A$ is a \textbf{SLIP} algebra.
\end{proof}
By the following theorems we will see the condition that $\A_{kk}$ ($1\leq k \leq n-1$) be \textbf{SLIP} algebras is not necessary for that $\A$ be a \textbf{SLIP} algebra.
\begin{thm}\label{gn2}
Suppose that the algebra $\A$ has a generalized triangular matrix representation 
\[  \A = \begin{pmatrix}
  \A_{11} & \A_{12} & \cdots & \A_{1n} \\
  0 & \A_{22} & \cdots & \A_{2n} \\
  \vdots & \vdots & \ddots  & \vdots \\
  0 & 0 & \cdots & \A_{nn}
\end{pmatrix}.\]
 Let any local left multiplier from $\A_{ii}$ ($2\leq i \leq n$) into $\A_{ki}$ ($1\leq k \leq i$) is a left multiplier, and for some $2\leq k \leq n$, $l.ann_{\A_{11}}\A_{1k}=\lbrace 0 \rbrace$. Then $\A$ is a \textbf{SLIP} algebra.
\end{thm}
\begin{proof}
Let $\A$ has a set of left triangulating idempotents  $\lbrace E_{1},\cdots, E_{n}\rbrace$. By Remark \ref{tt}, $\A\cong Tri(E_{1}\A E_{1}, E_{1}\A F, F\A F)$ and $\lbrace E_{2},\cdots, E_{n}\rbrace$ is a set of left triangulating idempotents of $F\A F$, where $F=I-E_{1}$. From hypothesis the algebra 
\[F\A F \cong \begin{pmatrix}
    \A_{22} & \cdots & \A_{2n} \\
   \vdots & \ddots  & \vdots \\
   0 & \cdots & \A_{nn}
\end{pmatrix}\]
satisfies in conditions of Theorem \ref{gn1}, and hence it is a \textbf{SLIP} algebra. A similar proof as that of Theorem \ref{gn1} shows that all local left multiplier from $F\A F$ into $E_{1}\A F$ is a left multiplier. Since for some $2\leq k \leq n$, $l.ann_{\A_{11}}\A_{1k}=\lbrace 0 \rbrace$, it follows that $l.ann_{\A_{11}} E_{1}\A F=\lbrace 0 \rbrace$. Now by Theorem \ref{suslip1}, $\A$ is a \textbf{SLIP} algebra.
\end{proof} 
\begin{thm}\label{gn3}
Let $\A$ has a set of left triangulating idempotents  $\lbrace E_{1},\cdots, E_{n}\rbrace$ with the generalized triangular matrix representation 
\[  \A = \begin{pmatrix}
  \A_{11} & \A_{12} & \cdots & \A_{1n} \\
  0 & \A_{22} & \cdots & \A_{2n} \\
  \vdots & \vdots & \ddots  & \vdots \\
  0 & 0 & \cdots & \A_{nn}
\end{pmatrix}.\]
Let $l.ann_{F\A F}F \A E_{n}=\lbrace 0 \rbrace$, where $F=I-E_{n}$ and any local left multiplier from $\A_{nn}$ into $\A_{kn}$ ($1\leq k \leq n$) is a left multiplier. Then $\A$ is a \textbf{SLIP} algebra.
\end{thm}
\begin{proof}
We use the same notations as that in proof of Theorem \ref{gn1}. Let $F=I-E_{n}$. By Remark \ref{tt}, $\A\cong Tri(F\A F, F\A E_{n}, E_{n}\A  E_{n})$. In this case $E_{n}$ is the unity of $\A_{nn}$ and we have the $\rr$-algebra isomorphism:
\[E_{n}\A E_{n}\cong \A_{nn},\]
and the $(F\A F, \A_{nn})$-bimodule isomorphism:
\[F\A E_{n} \cong \begin{pmatrix}
\A_{1n}\\
\vdots \\
\A_{n-1,n}
\end{pmatrix}.\]
By hypothesis $\A_{nn}$ is a \textbf{SLIP} algebra. Let $\psi: \A_{nn} \rightarrow F\A E_{n}$ be a local left multiplier, we show that $\psi$ is a left multiplier. 
\par 
Since $\psi$ is an $\rr$-linear map, there exist $\rr$-linear maps $\phi_{n}^{k}:\A_{nn}\rightarrow \A_{kn}$ ($1\leq k \leq n-1$) such that
\[\psi(a_{nn}E_{n})=\sum_{k=1}^{n-1}\phi_{n}^{k}(a_{nn})E_{kn}.\]
For each $a_{nn}\in \A_{nn}$, let $T=a_{nn}E_{n}$. Since $\psi$ is local left multiplier, there is $X_{T}=\sum_{k=1}^{n-1}x_{kn}E_{kn}\in F\A E_{n}$ such that $\psi(T)=X_{T}T$. So 
\[ \sum_{k=1}^{n-1}\phi_{n}^{k}(a_{nn})E_{kn}=\sum_{k=1}^{n-1}x_{kn}a_{nn}E_{kn},\]
and hence any $\phi_{n}^{k}$ is a local left multiplier. From hypothesis each $\phi_{n}^{k}$ ($1\leq k \leq n-1$) satisfies
\[\phi_{n}^{k}(a_{nn})=\phi_{n}^{k}(1_{n})a_{nn} \quad (a_{nn} \in \A_{nn}).\]
So 
\[ \psi(a_{nn}E_{n})=\sum_{k=1}^{n-1}\phi_{n}^{k}(1_{n})a_{nn}E_{kn}=\psi(E_{n})a_{nn}E_{n},\]
for all $a_{nn} \in \A_{nn}$. So $\psi$ is a left multiplier. Now by hypothesis and Theorem \ref{suslip1}, it follows that $\A$ is a \textbf{SLIP} algebra.
\end{proof}
The following proposition shows that the condition $E_{n}\A E_{n}$ be a \textbf{SLIP} algebra is a necessary condition for the result that the algebra $\A$ with a set of left triangulating idempotents  $\lbrace E_{1},\cdots, E_{n}\rbrace$ be a \textbf{SLIP} algebra.
\begin{prop}\label{necgn}
If $\A$ is a \textbf{SLIP} algebra which has a set of left triangulating idempotents $\lbrace E_{1},\cdots, E_{n}\rbrace$, then $E_{n}\A E_{n}$ is a \textbf{SLIP} algebra.
\end{prop}
\begin{proof}
Let $F=I-E_{n}$. By Remark \ref{tt}, $\A\cong Tri(F\A F, F\A E_{n}, E_{n}\A  E_{n})$. By Theorem \ref{nec}, it follows that $E_{n}\A E_{n}$ is a \textbf{SLIP} algebra.
\end{proof}
In continue we apply our results to block upper triangular matrix algebras. In order to prove Theorem \ref{bl}, we need the following lemma. 
\begin{lem}\label{lll}
Let $\A$ be a \textbf{SLIP} algebra and $M_{r\times s}(\A)$ be the set of all $r$-by-$s$ matrices over $\A$ as a right $\A$-module. Then any local left multiplier from $\A$ into $M_{r \times s}(\A)$ is a left multiplier.
\end{lem}
\begin{proof}
We use the same notations as that in proof of Theorem \ref{gn1}. Let $\psi: \A \rightarrow M_{r\times s}(\A)$ be a local left multiplier. Then there exist $\rr$-linear maps $\phi^{ij}:\A\rightarrow \A E_{ij}$ ($1\leq i\leq r, 1\leq j \leq s$) such that
\[\psi(a)=\sum_{i=1}^{r}\sum_{j=1}^{s}\phi^{ij}(a)E_{ij}.\]
Since $\psi$ is a local left multiplier, it is easily checked that each $\phi^{ij}$ is a local left multiplier. So by hypothesis and the fact that $\A E_{ij}\cong \A$ (as right $\A$-module) we have 
\[\phi^{ij}(a)=\phi^{ij}(1)aE_{ij},\]
for all $1\leq i\leq r, 1\leq j \leq s$. Thus $\psi(a)=\psi(1)a$ for all $a\in\A$ i.e., $\psi$ is a left multiplier.
\end{proof}
\begin{thm}\label{bl}
Let $B_{n}^{\bar{k}} (\A)$ ($n\geq 1$) be a block upper triangular matrix algebra, where $\bar{k} = (k_{1},\cdots, k_{m})\in \mathbb{N}^{m}$. Then we have the following.
\begin{enumerate}
\item[(i)] If $k_{m}\geq 2$, then $B_{n}^{\bar{k}} (\A)$ ($n\geq 2$) is a \textbf{SLIP} algebra.
\item[(ii)] If If $k_{m}=1$, then $B_{n}^{\bar{k}} (\A)$ is a \textbf{SLIP} algebra
if and only if $\A$ is a \textbf{SLIP} algebra.
\end{enumerate}
\end{thm}
\begin{proof}
Let $\{F_{1},\ldots,F_{m}\}$ be a set of left triangulating idempotents of $B_{n}^{\bar{k}} (\A)$ such that $F_{1}=\sum_{i=1}^{k_{1}}E_{i}$ and $F_{j}=\sum_{i=1}^{k_{j}}E_{i+k_{1}+\cdots+k_{j-1}}$ for $2\leq j\leq m$. Suppose that $ F=I-F_{m}$ and $l=k_1+...+k_{m-1}$. Then $F B_{n}^{\bar{k}} (\A) F$ is a subalgebra of $M_{l}(\A)$ and $F B_{n}^{\bar{k}} (\A) F_{m}\cong M_{l\times k_{m}}$. Since $l.ann_{M_{l}(\A)}M_{l\times k_{m}}=\lbrace 0 \rbrace$, it follows that $l.ann_{F B_{n}^{\bar{k}} (\A) F}F B_{n}^{\bar{k}} (\A) F_{m}=\lbrace 0 \rbrace$.
\par 
$(i)$  We have $F_{m}B_{n}^{\bar{k}} (\A)F_{m}\cong M_{k_{m}}(\A)$ ($k_{m}\geq 2$). By Corollary \ref{mn}-$(i)$, any local left multiplier from $F_{m}B_{n}^{\bar{k}} (\A)F_{m}$ into any right $F_{m}B_{n}^{\bar{k}} (\A)F_{m}$-module is a left multiplier. From Theorem \ref{gn3} we see that $\A$ is \textbf{SLIP}.
\par 
$(ii)$ Let $B_{n}^{\bar{k}} (\A)$ be a \textbf{SLIP} algebra. By Proposition \ref{necgn}, $F_{m}B_{n}^{\bar{k}} (\A) F_{m}\cong M_{k_{m}}(\A)=\A$ (since $k_{m}=1$) is a \textbf{SLIP} algebra. 
\par
 Conversely, assume that $\A$ is a \textbf{SLIP} algebra. We have $F B_{n}^{\bar{k}} (\A) F_{m}\cong M_{l\times 1}$ (in this case $F_{m}=E_{n}$). Now by Lemma \ref{lll}, each of Theorems \ref{suslip1} and \ref{gn3} implies that $B_{n}^{\bar{k}} (\A)$ is a \textbf{SLIP} algebra.
\end{proof}
The $n$-by-$n$ upper triangular matrices $T_{n}(\A)$ ($n\geq 1$) is the block upper triangular matrix algebra $B_{n}^{\bar{k}} (\A)$, whenever $\bar{k} =(k_{1},\cdots, k_{n})\in \mathbb{N}^{n}$ with $k_{j}=1$ for any $1\leq j\leq n$. Therefore, by Theorem \ref{bl}, we obtain the following corollary.
\begin{cor}
The algebra of upper triangular matrices $T_{n}(\A)$ ($n\geq 1$) is a \textbf{SLIP} algebra if and only if $\A$ is a \textbf{SLIP} algebra.
\end{cor}
Since each $\rr$-linear local left multiplier from $\rr$ into any $\rr$-module $\X$ is a left multiplier, from Theorem \ref{bl}, the next corollary is immediate.
\begin{cor}
The block upper triangular matrix algebra $B_{n}^{\bar{k}} (\rr)$ is a \textbf{SLIP} algebra for every $n\geq 1$. Especially, $T_{n}(\rr)$ ($n\geq 1$) is a \textbf{SLIP} algebra.
\end{cor}


\bibliographystyle{amsplain}
\bibliography{xbib}

\begin{thebibliography}{20}

\bibitem{An} R. An and J. Hou, \textit{Characterizations of derivations on triangular rings: additive maps derivable at idempotents}, Linear Algebra Appl. 431 (2009), 1070-1080.

\bibitem{Ben2} D. Benkovi$\check{\textrm{c}}$ and D. Eremita, \textit{Commuting traces and commmutativity preserving maps on triangular algebras}, J. Algebra, 280 (2004) 797-824.

\bibitem{Bir1} G. F. Birkenmeier, \textit{Idempotents and completely semiprime ideals}, Comm. Algebra 11 (1983) 567--580.

\bibitem{Bir2} G.F. Birkenmeier, H. E. Heatherly, J.Y. Kim, J.K. Park, \textit{Triangular matrix representations}, J. Algebra 230 (2000) 558--595.

\bibitem{br1} M. Bre$\breve{s}$ar and P. $\breve{S}$emrl, \textit{Mappings which preserve idempotents, local automorphisms, and local derivations}, Canad. J. Math. 45 (1993), 483--496.

\bibitem{br2} M. Bre$\breve{s}$ar, \textit{Characterizing homomorphisms, derivations and multipliers in rings with
idempotents}, Proc. Roy. Soc. Edinburgh Sect. A 137 (2007), 9--21.

\bibitem{bre11} M. Bre$\check{\textrm{s}}$ar, M. Gra$\check{\textrm{s}}$i$\acute{c}$ and J. S. Ortega, \textit{Zero product determined matrix algebras}, Linear Algebra Appl. 430 (2009), 1486--1498.

\bibitem{bre3} M. Bre$\check{\textrm{s}}$ar, \textit{Multiplication algebra and maps determined by zero products}, Linear and Multilinear Algebra, 60 (2012), 763--768.

\bibitem{Ch} S.U. Chase, \textit{A generalization of the ring of triangular matrices}, Nagoya Math. J. 18 (1961) 13-25.

\bibitem{ful1} K. R. Fuller, W. K. Nicholson and J. F. Watters, \textit{Reflexive bimodules}, Canad. J. Math. 41 (1989), 592--611.

\bibitem{ful2} K .R. Fuller, W. K. Nicholson and J. F. Watters, \textit{Universally reflexive algebras}, Linear Algebra Appl. 157 (1991), 195--201.

\bibitem{ful3} K .R. Fuller, W. K. Nicholson and J. F. Watters, \textit{Algebras whose projective modules are reflexive}, J. Pure Appl. Algebra 98 (1995), 135--150.

\bibitem{gh2} H. Ghahramani, \textit{Zero product determined triangular algebras}, Linear Multilinear Algebra, 61 (2013), 741--757.

\bibitem{gh3} H. Ghahramani, \textit{On rings determined by zero products}, J. Algebra and appl. 12 (2013), 1--15.

\bibitem{gh} H. Ghahramani, \textit{On derivations and Jordan derivations through zero products}, Operator and Matrices, 8 (2014), 759--771.

\bibitem{had1} D. W. Hadwin, \textit{Algebraically reflexive linear transformations}, Linear Multilinear Algebra, 14 (1983), 225--233.

\bibitem{had10} D. W. Hadwin and J. W. Kerr, \textit{Scalar-reflexive rings}, PAMS 103 (1988), l--7.

\bibitem{had11} D. W. Hadwin and J. W. Kerr, \textit{Scalar-reflexive rings II}, J. Algebra 125 (1989), 311--319.

\bibitem{had3} The Hadwin Lunch Bunch, \textit{Local multiplications on algebras spanned by idempotents}, Linear and Multilinear Algebra, 37 (1994), 259--263.

\bibitem{had} D. Hadwin and J. Kerr, \textit{Local multiplications on algebras}, J. Pure Appl. Algebra, 115 (1997), 231--239.

\bibitem{had2} D. Hadwin and J. Li, \textit{Local derivations and local automorphisms}, J. Math. Anal. Appl. 290 (2004), 702--714.

\bibitem{had4} D. Hadwin and J. Li, \textit{Local derivations and local automorphisms on some algebras}, J. Operator Theory, 60 (2008), 29--44.


\bibitem{john} B. E. Johnson, \textit{Centralisers and operators reduced by maximal ideals}, J. London Math. Soc. 43 (1968), 231--233.

\bibitem{ka} E. Katsoulis, \textit{Local maps and the representation theory of operator algebras}, Trans. Amer. Math. 368 (2016), 5377--5397.


\bibitem{li} J. Li and Z. Pan, \textit{Annihilator-preserving maps, multipliers, and derivations}, Linear Algebra Appl. 432 (2010), 5--13.

\bibitem{liu} D. liu and J. H. Zhang, \textit{Jordan Higher derivable maps on triangular algebras by commutative zero products}, Acta Math. Sinica, English series, 32 (2016), 258--264.

\bibitem{sn1} N. Snashall, \textit{Two questions on scalar-reflexive rings}, PAMS (1992).

\bibitem{sn2} N. Snashall, \textit{AlgLat for modules over fsi rings}, J. Algebra (1993).


\bibitem{sn3} N. Snashall, \textit{Scalar-reflexivity and FGC rings}, J. Algebra (1994).

\bibitem{Zh0} J. H. Zhang, A. L. Yang and F. F. Pan, \textit{Linear maps preserving zero products on nest subalgebras of von Neumann algebras}, Linear Algebra Appl. 412 (2006) 348-361.

\bibitem{Zha} S. Zhao and J. Zhu, \textit{Jordan all-derivable points in the algebra of all upper triangular matrices}, Linear Algebra Appl, 433 (2010) 1922-1938.

\end{thebibliography}

\end{document}